\documentclass[9pt,shortpaper,twoside,web]{ieeecolor}
\usepackage{generic}
\usepackage{cite}
\usepackage{amsmath,amssymb,amsfonts}
\usepackage{algorithmic}
\usepackage{graphicx}
\usepackage{textcomp}

\usepackage{graphicx}
\usepackage{textcomp}
\usepackage{xcolor}
\usepackage{cite}
\usepackage{cases}
\usepackage{mathrsfs}
\usepackage{amscd}
\usepackage{graphicx}  
\usepackage{epstopdf}
\usepackage{subfigure}
\usepackage[justification=centering]{caption}
\usepackage{diagbox}
\usepackage{multirow}
\usepackage{booktabs}
\renewcommand \thetable{\arabic{table}}
\usepackage{epstopdf}
\usepackage{bm}
\allowdisplaybreaks
\usepackage{booktabs}
\usepackage{multirow}
\usepackage{stfloats}
\newtheorem{theorem}{Theorem}

\newtheorem{remark}{Remark}[section]

\usepackage{mathrsfs}
\usepackage{mathtools}

\newcommand{\R}{\mathbb R}
\renewcommand{\t}{^{\mbox{\tiny\sf T}}}

\renewcommand{\t}{^{\mbox{\tiny\sf T}}}

\newcommand{\mE}{\mathbb{E}}
\newcommand{\V}{\mathcal{V}}

\newcommand{\G}{\mathcal{G}}
\newcommand{\N}{\mathcal{N}}
\newcommand{\A}{\mathcal{A}}
\renewcommand{\L}{\mathcal{L}}
\newcommand{\E}{\mathcal{E}}

\begin{document}
\title {Winners Take All: A Reverse Consensus Model }
%\markboth{}
%

\author{Zhiyong Chen 
	\thanks{Z. Chen is with the School of Engineering, The University of Newcastle, Callaghan, NSW 2308, Australia, Tel: +61249216352, Fax: +61249216993.   
		E-mail:  zhiyong.chen@newcastle.edu.au.}}

 \maketitle
	
		%%%%%%%%%%%%%%%%%%%%%%%%%%%%%%%%%%%%%%%%%%%%%%%%%%%%%%%%%%%%%%%%%%%%%%%%%%%%%%%%
	
\begin{abstract} 
This paper introduces a nonlinear multi-agent dynamic model that characterizes the resource-seizing mechanism for a fixed amount of resources. The model demonstrates a winners-take-all behavior within a zero-sum game framework. It represents one of the simplest dynamics where equilibria correspond to states of winners and losers, with every trajectory converging to such an equilibrium. Notably, when the model operates in reverse time, it resembles a multi-agent consensus model, referred to as a reverse consensus model. The key characteristics of this model are explored through rigorous analysis.
 \end{abstract}
	

\section{Introduction}

Developing simple models to describe complex behavior or phenomena is essential because it reduces systems to their key components, making them easier to understand and analyze. The consensus model in multi-agent systems is one such model, with widespread contributions across various fields. Its origins trace back to social and behavioral sciences, particularly social network analysis, with the French-DeGroot model \cite{french1956formal,degroot1974reaching} being a foundational example. In physics, consensus models have also been explored through simulations \cite{reynolds1987flocks,vicsek1995novel} and given theoretical backing in \cite{jadbabaie2003coordination}. These models, along with their advanced variants, have attracted significant research interest for decades.

Average consensus specifically seeks for all agents to agree on the average of their initial values, representing pure cooperation. This is conceptually similar to zero-sum games, where resources are fixed and allocated among players. However, in zero-sum games, players primarily compete for these resources. Competition is crucial in the development of many complex behaviors. In competitive scenarios, winners-take-all dynamics occur when some players secure all the resources while others receive nothing, maintaining a constant total resource. This paper introduces and explores a simple multi-agent model that captures this phenomenon and examines its relationship with the consensus model.

Winners-take-all behaviors appear in various contexts. In free markets, companies may vie for the same customer base, leading to dominance by one due to factors like economies of scale or network effects, resulting in a monopolistic scenario \cite{noe2005winner}. Similarly, social media platforms compete for user attention, with popularity and value increasing for the platform that attracts more users \cite{bourai2023winner}. In ecosystems, species compete for limited resources, which can lead to one species becoming dominant in its niche \cite{kocher2023darwinian}.

A classical mathematical model for competitive dynamics is the Lotka-Volterra model, first proposed by Lotka \cite{lotka1925elements} and Volterra \cite{volterra1927variations}. Originally developed to describe predator-prey interactions in ecological systems, this model consists of a set of differential equations that capture how predator and prey populations affect each other.  The Lotka-Volterra model can also be adapted and generalized to study broader competition and cooperation dynamics, as explored in works such as \cite{hirsch1988systems,zeeman1998three}.
 
 In \cite{maurer2003competitive}, the Lotka-Volterra model is generalized to describe the dynamics of website growth. The model demonstrates that as competition between agents intensifies, a sudden transition occurs from a regime where many agents coexist to a winners-take-all equilibrium, where a few agents capture nearly all the resources while most others become nearly extinct. This equilibrium is driven by a balance between an agent’s growth and the losses due to competition, differing from the pure resource-seizing mechanism found in zero-sum games.
 
The strength of competition is characterized by a constant in \cite{maurer2003competitive}, which was later extended to a linear form, giving rise to the ``rich-get-richer" effect \cite{yanhui2007competitive,yanhui2006qualitative}. Further generalizations in \cite{caram2010dynamic,caram2014complex} introduced a nonlinear interaction function between agents, where the interaction strength depends on the difference in agent sizes. This formulation can capture both competitive and cooperative dynamics, leading to agent clustering in the stationary state. Despite these extensions, the fundamental mechanism for achieving equilibrium, including the winners-take-all outcome, remains unchanged. However, this paper aims to develop a new model to capture the winners-take-all mechanism specifically for pure resource-seizing behavior in a zero-sum game.

It is worth noting that the winners-take-all model is formulated differently in some related works, where it describes competition among a group of agents, with only the one receiving the largest input remaining active while all others are deactivated. For instance, in \cite{binas2014learning}, the model consists of excitatory units connected to a common inhibitory unit, which provides recurrent inhibitory feedback. The excitatory unit receiving the strongest input suppresses all other units via this feedback loop, effectively ``winning'' the competition.  A similar process plays a role in stable long-term memory formation, as discussed in \cite{sajikumar2014competition}. A theoretical analysis of this winners-take-all mechanism is presented in \cite{li2016distributed}, with an extension to allow multiple winners, corresponding to the top $k$ inputs, in \cite{liang2023distributed}. Unlike the mechanism studied in this paper, these models determine winners based solely on the magnitude of inputs. 
 
 In the remainder of this paper, the new model is introduced in Section~\ref{sec:model}. It features a simple nonlinear multi-agent dynamic that effectively captures winners-take-all behavior within a zero-sum game framework, where agents compete for a fixed amount of resources. A rigorous analysis is presented in this section, demonstrating that every trajectory converges to a winners-take-all equilibrium. Interestingly, when the model is considered in reverse time, it resembles a multi-agent consensus model, referred to as a reverse consensus model. Numerical simulations are provided in Section~\ref{sec:simu}, while Section~\ref{sec:conclusion} concludes the paper with a discussion of future research directions.

Some of the notations used in this paper are defined as follows. A undirected graph $\G=\{\V,\E\}$ is considered for the network studied in this paper, where $\V=\{1,\ldots,n\}$ denotes the set of nodes and $\E\subset\V\times\V$ represents the set of edges.  Additionally, no self-loops exist, i.e., $(i,i)\notin \E,\;\forall i\in \mathcal V$.  The weighted adjacency matrix  $\A \in\R^{n \times n}$ has an $(i,j)$-entry $a_{ij}>0$ if $(j,i)\in\E$ and $a_{ij}=0$ if $(j,i)\notin\E$. For a subset $\bar \V \subseteq \V$, define an operator $\wedge$ as $\E \wedge   \bar \V  =\{(i,j)\in \E \;|\; i, j \in \bar \V\} $ and $\G \wedge   \bar \V = \{ \bar \V, \E \wedge   \bar \V \}$. For a set $\V$, $|\V|$denotes its cardinality.  The set difference between two sets $\V_1$ and $\V_2$ is written as $\V_1 \setminus \V_2$. That is, $i \in \V_1 \setminus \V_2$ means that $i \in \V_1$ and $i \notin \V_2$. The $n$-dimensional set of non-negative real numbers is denoted as  $\R^n_+$, which can be represented as:
\begin{equation*}
\R^n_+= \{ x=[x_1, x_2, \dots, x_n]\t \mid x_i \geq 0  \text{ and } x_i \in \R, \, \forall i \in \V \}.
\end{equation*}

 \section{The Model}
 
 \label{sec:model}

The winner-takes-all model is defined by the following nonlinear multi-agent dynamics:
 \begin{align}
 \dot x_i = \sum_{j \in \N_i} a_{ij} (x_i-x_j) x_i x_j,\;\forall i\in \V \label{model}
 \end{align}
where $x=[x_1, x_2, \dots, x_n]\t \in \R^n_+$ with $x_i(t)$ representing the state of each agent at time $t$. The time variable $t$ is omitted for simplicity.  It is clear that $\R^n_+$ is an invariant set. Specifically, if the initial state satisfies $x(0) \in \R^n_+$, then the states remain within $x_i(t) \in \R^n_+$ for all $t \geq 0$. This is because when $x_i(t) = 0$, the dynamics $\dot{x}_i(t) = 0$ ensure that the state remains at the boundary. In this section, we will explore several key characteristics of the model.

The state value $x_i$ of an agent represents the amount of resource it occupies. A higher resource level indicates a stronger agent. A stronger agent, $x_i$, takes resources from a weaker agent, $x_j$, when a connection exists between them with $a_{ij} > 0$, modeled by the change in $x_i$ in the direction of $x_i - x_j$. This transfer of resources reflects competitive behavior. The intensity of the competition is proportional to the strengths of both agents, $x_i$ and $x_j$, and is represented by their product, $x_i x_j$. This dynamic is captured by the model \eqref{model}.

First, the following theorem demonstrates that the model in \eqref{model} represents a zero-sum game. In this context, the gains of some agents are exactly offset by the losses of others. In other words, the total sum of benefits and losses across all agents equals zero, resulting in no net gain or loss for the multi-agent system as a whole.

\begin{theorem}
The total state of the system in \eqref{model} remains constant over time. Specifically,  $\sum_{i \in \V} x_i(t)=\sum_{i \in \V} x_i(0),\, \forall t\geq 0$.
\end{theorem}
 
\begin{proof} The proof follows from the direct calculation:
  \begin{align*}
& \sum_{i \in \V} \dot x_i = \sum_{i \in \V} \sum_{j \in \N_i} a_{ij} (x_i-x_j) x_i x_j \\
 =&\frac{1}{2}\sum_{(i,j) \in \E} [ a_{ij} (x_i-x_j) x_i x_j +a_{ji} (x_j-x_i) x_j x_i ]=0
 \end{align*}
 for all $t \geq 0$. This result is due to $a_{ij}=a_{ji}$  in an undirected network. Thus, $\sum_{i \in \V} x_i(t)$ remains constant over time.\end{proof}

The winner-take-all feature is demonstrated in the following theorem, which describes the attractiveness of the set defined as:
 \begin{align} \label{mE}
\mE= \left\{ 
 \left.
\begin{array}{l}
x_i > 0, \, i\in \V_w\\
x_i = 0, \, i\in \V_l  
\end{array} \right|
\begin{array}{l}
x_i=x_j,\, \forall (i,j)\in \E \wedge   \V_w  \\
\V_w \cup \V_l =\V,\, \V_w \cap \V_l =\emptyset
\end{array}
 \right  \}.
 \end{align}
In this set, $\V_w$ represents the ``winner'' agents who take resources, indicated by their positive values, while $\V_l$ represents the ``loser'' agents, who take nothing, with values of zero. There may be more than one winner; however, winners either engage in direct competition and settle at the same value to reach equilibrium (which is later shown to be unstable) or are not directly connected in competition. 
 
 \begin{theorem} \label{theorem2}
Every element of $\mE$ is an equilibrium point of \eqref{model}, and every trajectory of \eqref{model} converges to an equilibrium point in $\mE$ as $t\rightarrow\infty$. 
\end{theorem}

\begin{proof} We first prove that every element of $\mE$ is an equilibrium point of \eqref{model}. This holds because, at every element of $\mE$, we have $\dot{x} = 0$, meaning the agents' states do not change. More specifically, for $x_i = 0$ with $i \in \V_l$, it is clear from \eqref{model} that $\dot{x}_i = 0$. For $x_i > 0$ with $i \in \V_w$, we compute:
 \begin{align*}
 \dot x_i = \sum_{j \in \N_i \cap \V_l} a_{ij} (x_i-x_j) x_i x_j +\sum_{(i,j)\in \E \wedge   \V_w } a_{ij} (x_i-x_j) x_i x_j. 
 \end{align*}
The first term is zero because $x_j = 0$, and the second term is zero because $x_i = x_j$.

 %In the proof, we consider the graph is connected. For a graph that is not connected, we apply the same proof on every connected subgraph, or trivially one isolated agent, to reach the conclusion for the whole graph. 

Below, we aim to prove the second part of the statement: the convergence of every trajectory to an equilibrium point in $\mE$. For each trajectory, we define the set $\hat\V$ as the set of agents whose states converge to zero, meaning 
\begin{equation}
\lim_{t \rightarrow \infty} x_i(t) = 0,\;\forall i \in \hat\V,
\label{hatV}
\end{equation} 
and this limit does not hold for $i \notin \hat\V$. Let $\bar\V = \V \setminus \hat\V$.

If $\bar\V = \emptyset$, the trajectory converges to $x^* = 0 \in \mE$ as $t\rightarrow\infty$, and the proof is complete. If $\bar \V \neq \emptyset$, we can decompose it as $\bar \V = \bar \V_1 \cup \cdots \cup \bar \V_m$ for $m\geq 1$, where each subgraph $\G \wedge \bar \V_k$, for $k=1, \dots, m$, is connected, and $\bar\V_{k_1}\cap \bar\V_{k_2}=\emptyset$ for for $k_1 \neq k_2$. Moreover, for every edge $(i, j) \in \E \wedge \bar\V$, there exists some $k = 1, \dots, m$ such that $i, j \in \bar \V_k$. In other words, there is no edge $(i, j) \in \E \wedge \bar \V$ where $i \in \bar \V_{k_1}$ and $j \in \bar \V_{k_2}$ for $k_1 \neq k_2$.

If, for every $k = 1, \dots, m$, $\lim_{t\rightarrow\infty} x_i(t) =c_k, \forall i \in \bar \V_k$, holds with a constant $c_k > 0$, we can define $\V_w = \bar \V$, $\V_l = \hat \V$, and
  \begin{align*}
 &x^*=[x^*_1, x^*_2, \dots, x^*_n]\t \\
 &x_i^* =c_k,\; \forall i \in \bar \V_k,\; k=1,\cdots,m \\
 &x_i^* =0,\;\forall i \in \hat\V.
 \end{align*}
Because for every edge $(i, j) \in \E \wedge \bar \V$, there exists some $k = 1, \dots, m$ such that $i, j \in \bar \V_k$, we have   $x_i^*=x_j^*=c_k,\, \forall (i,j)\in \E \wedge   \V_w $. As a result, $x^* \in \mE$, and $\lim_{t \rightarrow \infty} x(t) = x^*$, which completes the proof.
Otherwise, there exists some $\kappa$ such that
 \begin{equation}\label{hypo}
 \lim_{t\rightarrow\infty} x_i(t) =c_\kappa>0 ,\; \forall i \in \bar\V_\kappa,\; \text{does not hold}.
 \end{equation}
A contradiction will be derived for the hypothesis \eqref{hypo} below, considering two cases to complete the proof.

{\it Case 1:  $|\bar\V_\kappa| =1$.}
 
Denote the only agent in $\bar \V_\kappa$ as $\ell$, which has no neighbor in $\bar \V$. As a result, we have
   \begin{align}
 \dot x_\ell = \sum_{j \in \N_\ell \cap \hat\V} a_{\ell j} (x_\ell -x_j) x_\ell x_j. \label{dotxell}\end{align}
On one hand, since $\ell \notin \hat \V$, $x_\ell(t)$ does not converge to zero. Thus, there exists a consant $\epsilon > 0$ such that $x_\ell(t_i) > \epsilon$ for an infinite sequence $t_1 < t_2 < \cdots$ that approaches infinity. On the other hand, by the definition of $\hat \V$ in \eqref{hatV}, there exists a time $T$ such that 
$x_j(t)<\epsilon,\;\forall t \geq T,\; j \in \N_\ell \cap \hat\V$.  Pick a time $t_l > T$ such that $x_\ell(t_l) > \epsilon > x_j(t_l),\;\forall j \in \N_\ell \cap \hat \V$. Therefore, by \eqref{dotxell}, we have $\dot{x}_\ell(t) \geq 0$ for $t \geq t_l$. This implies that $x_\ell(t)$ is increasing for $t \geq t_l$ and must converge to a limit, denoted as $c_\kappa$. However, this contradicts the hypothesis \eqref{hypo}.

{\it Case 2:  $|\bar\V_\kappa| \geq 2$.}

 Denote $a_{\max}(t)$ as  the maximum value of the agents in $ \bar\V_\kappa$, that is 
  \begin{align}
a_{\max}(t) =  \max_{i\in \bar\V_\kappa} x_i(t). 
 \end{align}
 It is evident that $\lim_{t\rightarrow\infty}  a_{\max}(t)=0$ does not hold. Otherwise, $\lim_{t\rightarrow\infty} x_i(t) =0$ would hold for all $i \in \bar\V_\kappa$,  which contradicts the definition of $\bar\V_\kappa$. Therefore, there exists a constant $\epsilon>0$ such that $a_{\max} (t_i)>\epsilon$ for an infinite sequence $t_1<t_2<\cdots$ that approaches infinity. Additionally, by the definition of $\hat \V$ in \eqref{hatV}, there exists a time $T_1$ such that $x_j(t)<\epsilon,\;\forall t \geq T_1,\; j \in   \hat\V$. Pick  a time $t_l>T_1$ such that $a_{\max} (t_l)>\epsilon>x_j(t_l),\;\forall  j \in  \hat\V$.

Let $\mu'_1=t_l$. There exists either a finite sequence $\mu'_1 \leq \mu'_2  \leq \cdots \leq \mu'_{s'^*}$ for some finite constant $s'^*$ (with $\mu'_{s'^*+1}=\infty$ for notational completeness) or an infinite sequence $\mu'_1 \leq \mu'_2  \leq \cdots$ approaching infinity, such that during every interval $[\mu'_s, \mu'_{s+1}]$, for $s=1,\cdots, s'^*$ (in the finite case) or $s\geq 1$ (in the infinite case),  the same agent holds the maximum value.  The function $a_{\max}(t)$ is increasing for $t \geq \mu'_1$ if it is increasing within each interval $[\mu'_s, \mu'_{s+1}]$. Denote the agent with the maximum value as $\ell'_s$ within the interval $[\mu'_s, \mu'_{s+1}]$ and its dynamics can be written as 
 \begin{align}
 \dot x_{\ell'_s} = \sum_{j \in \N_{\ell'_s} \cap (\hat\V \cup \bar \V_\kappa )} a_{\ell'_s j} (x_{\ell'_s} -x_j) x_{\ell'_s} x_j.  \end{align}
We have $\dot x_{\ell'_s}\geq 0$ (i.e., $x_{\ell'_s} (t)=a_{\max} (t)$ is increasing) because $x_{\ell'_s} -x_j\geq 0$ holds for $j \in  \bar \V_\kappa$, as agent $\ell'_s$ has the maximum value; and it also holds for $j\in \hat\V$, since $a_{\max} (t_l)>\epsilon>x_j(t_l)$, and thus $x_{\ell'_s} (t)=a_{\max} (t)>\epsilon>x_j(t)$ during the interval. 
From this argument, we conclude that  $a_{\max}(t)$ is increasing  for $t\geq t_l$. As a result, 
\begin{equation}
\lim_{t\rightarrow\infty}  a_{\max}(t)=c_\kappa \label{amaxlim}
\end{equation}
for a constant $c_\kappa >0$. 

For an arbitrarliy small $\delta>0$, particularly chosen as 
\begin{equation}
\delta< c_\kappa (2^{|\bar\V_\kappa|}  +1)^{-1} 2^{|\bar\V_\kappa| -1},
\end{equation}
we set $\varepsilon =   2^{1- |\bar\V_\kappa|} \delta$.
As a result,  $\varepsilon< c_\kappa (2^{|\bar\V_\kappa|}  +1)^{-1}$,  and 
\begin{equation}
c_\kappa-2^{s} \varepsilon >\varepsilon,\; \forall 0\leq s< |\bar\V_\kappa|. \label{verepsilon}
\end{equation}
For this $\varepsilon>0$, due to \eqref{hatV} and \eqref{amaxlim},  there exists a time $T_2\geq T_1$ such that $a_{\max}(t) >c_\kappa-\varepsilon,\forall t\geq T_2$ and $x_j(t)<\varepsilon,\;\forall t \geq T_2,\; j \in   \hat\V$.

Let $\mu_1=T_2$. Denote $\ell_1 \in  \bar\V_\kappa$ as an agent taking the maximum value at $\mu_1$, though it may not be the only agent reaching this maximum value. Let $\mu_2\geq \mu_1$ be the first finite time instant, if exists, at which another agent $\ell_2 \in  \bar\V_\kappa$ reaches the same value of $\ell_1$.
Recursively, let $\mu_{s+1}\geq \mu_s$ be the first finite time instant, if exists, at which  another agent $\ell_{s+1} \in  \bar\V_\kappa$,
$\ell_{s+1} \notin \{\ell_1,\cdots, \ell_s\}$,  whose value $x_{\ell_{s+1}} (\mu_{s+1})$ reaches the value of $\min_{i\in \{\ell_1,\cdots, \ell_s\}} x_i(\mu_{s+1})$. 
There exist $s^*$ finite time instants $\mu_1,\cdots, \mu_{s^*}$ for some constant $1\leq s^*\leq |\bar\V_\kappa|$.  Denote $\mu_{s^*+1} =\infty$ for notational convenience.

At $\mu_1$, we have $x_{\ell_1} (\mu_1) > c_\kappa-\varepsilon >\varepsilon>x_j$ for $j \in   \hat\V$.
During the interval $t \in [\mu_1, \mu_2]$,  no agent in  $\bar\V_\kappa$ exceeds the value of $x_{\ell_1}$. 
According to    
 \begin{align*}
 \dot x_{\ell_1} = & \sum_{j \in \N_i \cap \bar\V_\kappa } a_{\ell_1 j} (x_{\ell_1} -x_j) x_{\ell_1} x_j  \\
& + \sum_{j \in \N_i \cap \hat\V} a_{\ell_1 j} (x_{\ell_1} -x_j) x_{\ell_1} x_j,
 \end{align*}
we have $\dot x_{\ell_1}(t) \geq 0$, ensuring that $x_{\ell_1} (t) > c_\kappa-\varepsilon$ for $t \in [\mu_1, \mu_2]$.

Suppose that for $s^* > s\geq 1$, the following holds:
 \begin{align}
\sum_{ i\in \{\ell_1,\cdots, \ell_s\} } x_{i} (t) >s c_\kappa-2^{s-1} \varepsilon ,\; t\in [\mu_{s}, \mu_{s+1}]. \label{sumxi}
 \end{align}
 This serves as the induction hypothesis, and it is true for $s=1$.

From \eqref{sumxi}, we can deduce
\begin{align*}
& \min_{i\in \{\ell_1,\cdots, \ell_s\}} x_i(\mu_{s+1}) \\
 \geq & \sum_{ i\in \{\ell_1,\cdots, \ell_s\} } x_{i} (\mu_{s+1})  -(s-1) c_\kappa 
 > c_\kappa-2^{s-1} \varepsilon.
\end{align*}
According to the definition of $\mu_{s+1}$,  at $\mu_{s+1}$, we have 
\begin{align*}
x_{\ell_{s+1}} ( \mu_{s+1}) =\min_{i\in \{\ell_1,\cdots, \ell_s\}} x_i(\mu_{s+1}) > c_\kappa-2^{s-1} \varepsilon.
\end{align*}
Thus, we can conclude
\begin{align}
\sum_{ i\in \{\ell_1,\cdots, \ell_{s+1}\} } x_{i} (\mu_{s+1}) >& [s c_\kappa-2^{s-1} \varepsilon] + [ c_\kappa-2^{s-1} \varepsilon] \nonumber \\
 = & (s+1) c_\kappa-2^{s} \varepsilon. \label{sumxi2}
 \end{align}

For $t \in [\mu_{s+1}, \mu_{s+2}]$, there exists no agent in  $\bar\V_\kappa\setminus  \{\ell_1, \cdots, \ell_{s+1} \}$ whose value exceeds $x_i$ for any $ i\in \{\ell_1,\cdots, \ell_{s+1}\}$. Therefore,
  \begin{align}
 & \sum_{i\in \{\ell_1, \cdots, \ell_{s+1} \}}\dot x_i \nonumber \\
 = & \sum_{i\in \{\ell_1, \cdots, \ell_{s+1} \}, j \in \N_i \cap ( \bar\V_\kappa \setminus  \{\ell_1, \cdots, \ell_{s+1} \}) } a_{ij} (x_i -x_j) x_i x_j \nonumber \\
 & + \sum_{i\in \{\ell_1, \cdots, \ell_{s+1} \}, j \in \N_i \cap \hat\V} a_{ij} (x_i -x_j) x_i x_j  \nonumber\\
 \geq & \sum_{i\in \{\ell_1, \cdots, \ell_{s+1} \}, j \in \N_i \cap \hat\V} a_{ij} (x_i -x_j) x_i x_j .     \label{dxis1}
 \end{align}
From \eqref{verepsilon} and \eqref{sumxi2}, we have
  \begin{align*}
x_{i} (\mu_{s+1}) &\geq \sum_{ i\in \{\ell_1,\cdots, \ell_{s+1}\} } x_{i} (\mu_{s+1})- s c_\kappa \\
& > c_\kappa-2^{s} \varepsilon >\varepsilon>x_j ,\;
\forall i\in \{\ell_1,\cdots, \ell_{s+1}\}, \; j \in   \hat\V.
  \end{align*}
Thus, together with \eqref{dxis1}, we conclude that
$\sum_{ i\in \{\ell_1,\cdots, \ell_{s+1} \} } x_{i} (t)$ is increasing for  $t \in [\mu_{s+1}, \mu_{s+2}]$. Consequently,
 \begin{align}
\sum_{ i\in \{\ell_1,\cdots, \ell_{s+1}\} } x_{i} (t) > 
(s+1) c_\kappa-2^{s} \varepsilon  ,\; t\in [\mu_{s+1}, \mu_{s+2}]. \label{sumxi3}
 \end{align}
 which matches  \eqref{sumxi} with $s$ replaced by $s+1$.

Using mathematical induction, we establish that \eqref{sumxi} holds for  $s=s^*$, meaning
  \begin{align}
\sum_{ i\in \{\ell_1,\cdots, \ell_{s^*}\} } x_{i} (t) >s^* c_\kappa-2^{s^*-1} \varepsilon  ,\; t \geq \mu_{s^*}  \label{sumxifinal}
 \end{align}
This implies that
  \begin{align}
 x_{i} (t) \geq & \sum_{ i\in \{\ell_1,\cdots, \ell_{s^*}\} } x_{i} (t) -[s^*-1]c_\kappa  \nonumber\\
 > & c_\kappa-2^{s^*-1} \varepsilon
  \geq    c_\kappa-2^{|\bar\V_\kappa|-1} \varepsilon \nonumber \\
= & c_\kappa-\delta ,\; t \geq \mu_{s^*}  ,\; i\in \{\ell_1,\cdots, \ell_{s^*}\}.  \label{xifinal}
 \end{align}
 As a result, there exists at least one agent (since $s^*\geq 1$), denoted as $\ell \in \bar \V_\kappa$, whose limit is $c_\kappa$, that is,
 $\lim_{t\rightarrow\infty}x_\ell(t)=c_\kappa$.
 
Next, we examine the behavior of every neighbor  $\hbar \in \bar\V_\kappa$ of $\ell$. Since all agent states are bounded, every agent has a bounded second-order derivative, denoted by $\ddot x_\ell(t)\leq b$.  For any small $\gamma>0$, particularly with $\gamma<c_\kappa/3$, select a sufficiently small $\sigma>0$ satisfying $\sigma<\gamma$ and
 \begin{align}
 & \frac{1}{a_{\ell \hbar} (c_\kappa -\sigma) }\left( 2\sqrt{b \sigma} + \sum_{j \in \N_\ell  \cap (\bar\V_\kappa/ \{\hbar\} ) } a_{\ell j} \sigma c_\kappa^2 \right) \nonumber\\
  <& (c_\kappa -\sigma) \gamma -\gamma^2 . \label{gamma}
 \end{align}
Such a small $\sigma$ always exists because the inequality \eqref{gamma} holds for $\sigma = 0$, and the expression is continuous in $\sigma$.
Denote $\rho= 2\sqrt{b \sigma}$. For this $\sigma>0$, there exists a finite time $\mu$ such that  $x_\ell(t) > c_\kappa -\sigma >x_j(t)$, $\forall j\in \hat \V$, $t\geq \mu$.
 
If there exists a time $t^* \geq \mu$ such that $\dot x_\ell(t^*) \geq \rho$, then $\dot x_\ell(t) \geq \rho/2$, $t\in [t^*, t^*+\rho/(2b)]$.  
Since $x_\ell(t^*) > c_\kappa -\sigma$, this implies $x_\ell(t^* +\rho/(2b)) > c_\kappa -\sigma +\rho^2/(4b)=c_\kappa$, which contradicts \eqref{amaxlim}.
Thus, we conclude that
\begin{equation}
\dot x_\ell(t) < \rho= 2\sqrt{b \sigma} ,\; \forall t\geq \mu.
\end{equation} 
 
  For $t\geq \mu$,  we have 
     \begin{align*}
  \dot x_\ell  
 =& \sum_{j \in \N_\ell  \cap (\bar\V_\kappa/ \{\hbar\} ) } a_{\ell j} (x_\ell-x_j) x_\ell x_j+
 a_{\ell \hbar} (x_\ell-x_\hbar) x_\ell x_\hbar \\
&  + \sum_{j \in \N_i \cap \hat\V} a_{\ell j} (x_\ell -x_j) x_\ell x_j  \nonumber\\
   \geq & - \sum_{j \in \N_\ell  \cap (\bar\V_\kappa/ \{\hbar\} ) } a_{\ell j} \sigma c_\kappa^2 +
 a_{\ell \hbar} (x_\ell-x_\hbar) x_\ell x_\hbar 
 \end{align*}
 Using \eqref{gamma}, it follows that
    \begin{align*}
& (x_\ell-x_\hbar)   x_\hbar  <\frac{1}{a_{\ell \hbar}x_\ell }\left( 2\sqrt{b \sigma}+ \sum_{j \in \N_\ell  \cap (\bar\V_\kappa/ \{\hbar\} ) } a_{\ell j} \sigma c_\kappa^2 \right)  \nonumber\\
   <& \frac{1}{a_{\ell \hbar} (c_\kappa -\sigma) }\left( 2\sqrt{b \sigma} + \sum_{j \in \N_\ell  \cap (\bar\V_\kappa/ \{\hbar\} ) } a_{\ell j} \sigma c_\kappa^2 \right) \\
   <&  (c_\kappa -\sigma) \gamma -\gamma^2 < x_\ell \gamma -\gamma^2.
 \end{align*}
 
 The above inequality 
   \begin{align*}
 (x_\ell-x_\hbar)   x_\hbar <x_\ell \gamma -\gamma^2  \end{align*}
is equivalent to 
  \begin{align*}
 (x_\hbar - (x_\ell-\gamma) ) (x_\hbar -\gamma)  >0. \end{align*}  
This results in the conclusion that either
  \begin{align*}
x_\hbar(t) > x_\ell(t)-\gamma > c_\kappa-\sigma-\gamma >c_\kappa-2\gamma ,\;\forall t\geq \mu
\end{align*}  
or 
  \begin{align*}
x_\hbar (t) <\gamma ,\;\forall t\geq \mu
\end{align*}  
Therefore,  every neighbor $\hbar \in \bar\V_\kappa$ of $\ell$  either converges to $c_\kappa$ or to zero.  This contradicts either the hypothesis \eqref{hypo} or the definition of $\hat \V$ in \eqref{hatV}, which states that no agent in $\bar \V_\kappa$ converges to zero.  \end{proof} 
 
Theorem~\ref{theorem2} states that every trajectory of \eqref{model} converges to an equilibrium point in the set $\mE$, where the system exhibits a winners-take-all behavior. Based on the definition of $\mE$ in \eqref{mE}, if there are multiple winners, they fall into two distinct categories: (i) winners who engage in direct competition and settle at the same value, and (ii) winners who are not directly connected and thus do not compete. More precisely, we can partition $\mE$ into two subsets, $\mE_u$ and $\mE_s$, which correspond to these categories:
  \begin{align*}
\mE_u= \left\{ 
 \left.
\begin{array}{l}
x_i > 0, \;  i\in \V_w\\
x_i = 0, \;  i\in \V_l  
\end{array} \right|
\begin{array}{l}
x_i=x_j,\; \forall (i,j)\in \E \wedge   \V_w \\
 \E \wedge   \V_w \neq \emptyset \\
 \V_w \cup \V_l =\V,\, \V_w \cap \V_l =\emptyset \end{array}
 \right  \}
 \end{align*}
and
 \begin{align*}
\mE_s= \left\{ 
 \left.
\begin{array}{l}
x_i > 0, \; i\in \V_w\\
x_i = 0, \; i\in \V_l  
\end{array} \right| 
\begin{array}{l}
\E \wedge   \V_w  = \emptyset \\
\V_w \cup \V_l =\V,\, \V_w \cap \V_l =\emptyset
\end{array}
\right \}.
 \end{align*}
In particular,  $\mE = \mE_u \cup \mE_s$. The following theorem demonstrates that every equilibrium point in $\mE_u$ is unstable. This implies that winners who engage in direct competition cannot maintain a stable settlement at the same value. Every trajectory of \eqref{model} generally converges to an equilibrium point in $\mE_s$.

\begin{theorem} \label{theorem3}
Every equilibrium point of \eqref{model} in $\mE_u$ is unstable.
\end{theorem}

\begin{proof} Consider any equilibrium point  $x^*=[x^*_1, x^*_2, \dots, x^*_n]\t \in \mE_u$.
According to the definition of $\mE_u$ where $\E \wedge   \V_w \neq \emptyset$, there exists a subset $\bar \V_w \subset \V_w$ such that $\G\wedge \bar \V_w$ is connected and $|\bar \V_w |\geq 2$. Additionally, there are no edges between an agent in $\bar \V_w$ and an agent in $\V_w\setminus \bar\V_w$.
 As a result, the equation \eqref{model} for the agents in $\bar\V_w$ becomes
   \begin{align}
 \dot x_i =& \sum_{j \in \N_i \cap \bar\V_w} a_{ij} (x_i-x_j) x_i x_j \nonumber\\
 & + \sum_{j \in \N_i \cap \V_l} a_{ij} (x_i-x_j) x_i x_j ,\;\forall i\in \bar\V_w. \label{xibarVw}
 \end{align}
 
According to definition of $\mE_u$, we have $x^*_i =c, \;  i\in \bar \V_w$ for a constant $c>0$ and $x^*_i = 0, \;  i\in \V_l$.
Define a neighbourhood of $x^*$ as follows: for a constant $\delta>0$,
\begin{align*}
   B_\delta=\left \{ x \in  \R_+^n \left|
   \begin{array}{l} 
   x_i \in (x^*_i-\delta, x^*_i+\delta ),  \;  i\in \V_w\\
   x_i = 0, \;  i\in \V_l 
   \end{array} \right. \right\}.
 \end{align*}
In $B_\delta$, the equation \eqref{xibarVw} reduces to
  \begin{align}
 \dot x_i =& \sum_{j \in \N_i \cap \bar\V_w} a_{ij} (x_i-x_j) x_i x_j  ,\;\forall i\in \bar\V_w, \label{xibarVw2}
 \end{align}
which is an autonomous subsystem completely decoupled from the original system \eqref{model}. 
Denote $\tilde x_i = x_i -x_i^*, \;\forall i\in \bar\V_w$. The system \eqref{xibarVw2} can be linearized to
  \begin{align}
 \dot {\tilde x}_i =& \sum_{j \in \N_i \cap \bar\V_w} a_{ij} \big[ (2x_i x_j -x_j^2)|_{x_i=x_j=c} \tilde x_i \nonumber\\
& +(x_i^2-2x_i x_j) |_{x_i=x_j=c} \tilde x_j \big] \nonumber\\
 =& c^2 \sum_{j \in \N_i \cap \bar\V_w} a_{ij}  (\tilde x_i -  \tilde x_j )  ,\;\forall i\in \bar\V_w. \label{xibarVw3}
 \end{align}
Clearly, the linear system \eqref{xibarVw3} describes the behavior of the (at least two) agents in $\bar\V_w$ where the network $\G\wedge \bar \V_w$ is connected. Consequently, the linear system matrix has a zero eigenvalue and the remaining eigenvalues are positive. Therefore, the system \eqref{xibarVw3} is unstable, and hence the original nonlinear system \eqref{xibarVw2} is also unstable. The proof is thus completed.   \end{proof}

\begin{remark}
Let $\tau=-t$ and $y_i(\tau)=x_i(-\tau)=x_i(t)$. The model \eqref{model} can be rewritten as 
 \begin{align}
 \frac{d y_i(\tau)}{d\tau} = -\sum_{j \in \N_i} a_{ij} (y_i(\tau)-y_j(\tau)) y_i(\tau) y_j(\tau),\;\forall i\in \V \label{reversemodel}
 \end{align}
which can also be expressed in compact form as
 \begin{align}
 \frac{d y(\tau)}{d\tau} = -\L(\tau) y(\tau) . \label{reversemodel2} \end{align}
The system \eqref{reversemodel} has a similar form to \eqref{model}, but with an additional negative sign. 
 The $(i,j)$-entry of the matrix $\L(\tau)$ is given by
 $l_{ij} (\tau) =-a_{ij} y_i(\tau) y_j(\tau)\leq 0$ for $i\neq j$ and 
 $l_{ii}(\tau)=-\sum_{j=1}^{n}l_{ij}(\tau) \geq 0$.  
Clearly, $\L(\tau)$ is a time-varying Laplacian matrix, and thus \eqref{reversemodel2} describes consensus behavior for $\tau \geq 0$ (or equivalently, $t \leq 0$ in reverse time) under a certain graph connectivity condition. The nonlinear term $x_i x_j$ (i.e., $y_iy_j$) does not contribute to the consensus behavior in \eqref{reversemodel}, but it plays a crucial role in the reverse consensus model \eqref{model}. Without this nonlinearity, a linear unstable system would yield divergent trajectories toward both positive and negative infinity, failing to capture the dynamics of resource-seizing mechanisms for a fixed amount of resources.
\end{remark}

 \begin{remark}
The model \eqref{model} is a simple representation of the dynamics described in the theorems of this section. It captures the essential characteristics while remaining a straightforward expression. This model can be adapted to more general forms to address specific features in different applications, while preserving its fundamental properties. Specifically, the term $x_i - x_j$ can be replaced with $f(x_i-x_j)$, where $f$ is any continuously differentiable odd function satisfying $f(a) > 0$ for all $a > 0$. Similarly, the term $x_i x_j$ can be replaced with $g(x_i, x_j)$, where $g$ is any continuously differentiable function such that $g(x_i, x_j) = g(x_j, x_i)$ and $g(0, a) = g(a, 0) = 0$ for all $a \geq 0$. Even with these modifications, the theorems still hold. 
 \end{remark}

\begin{figure}[t]
   \centering
 \includegraphics[width=7cm]{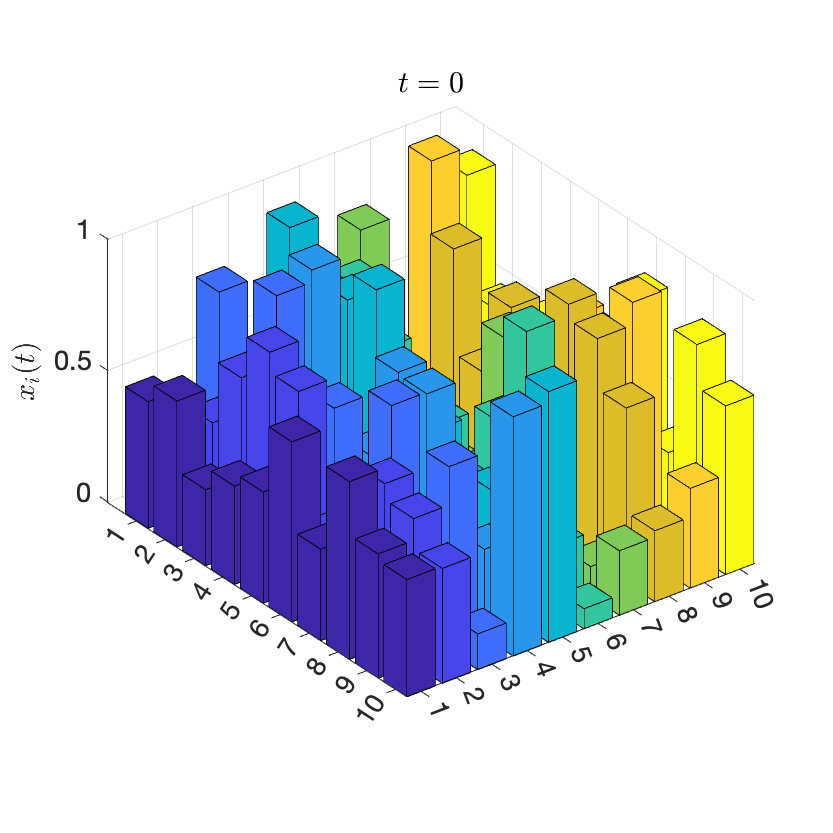} \\
 \vspace*{-7mm}
  \includegraphics[width=7cm]{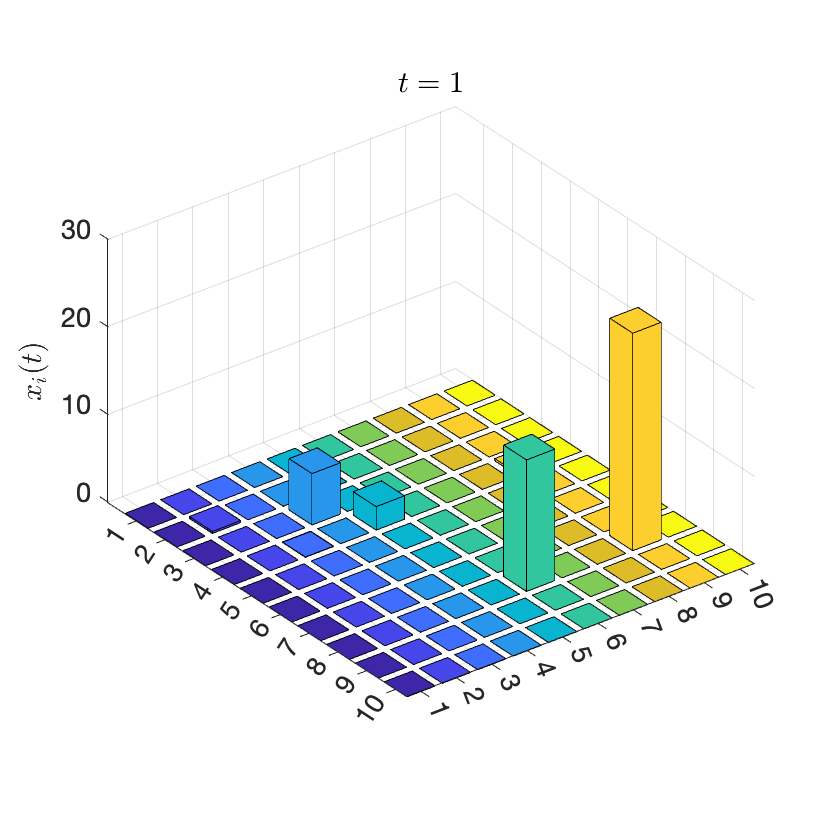}\\
   \vspace*{-7mm}
 \caption{The states of 100 agents at the initial time (top graph) and at the final time (bottom graph).}
  \label{fig.100agents1}
\end{figure}

\begin{figure}[t]
   \centering
 \includegraphics[width=9.5cm]{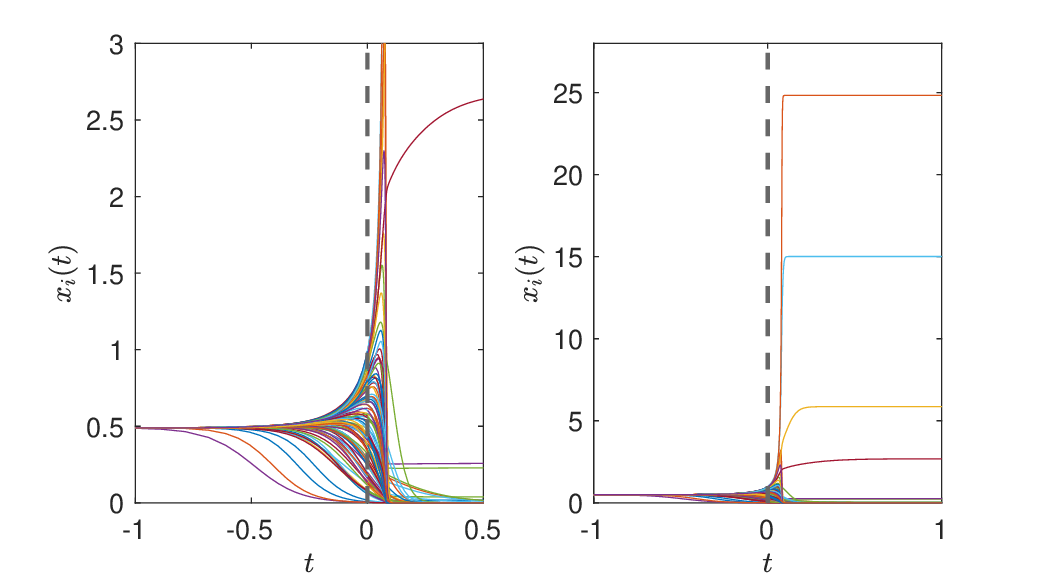} \\
 \caption{The trajectories of 100 agents over positive and negative time. }
  \label{fig.100agents3}
    \centering
 \includegraphics[width=9.5cm]{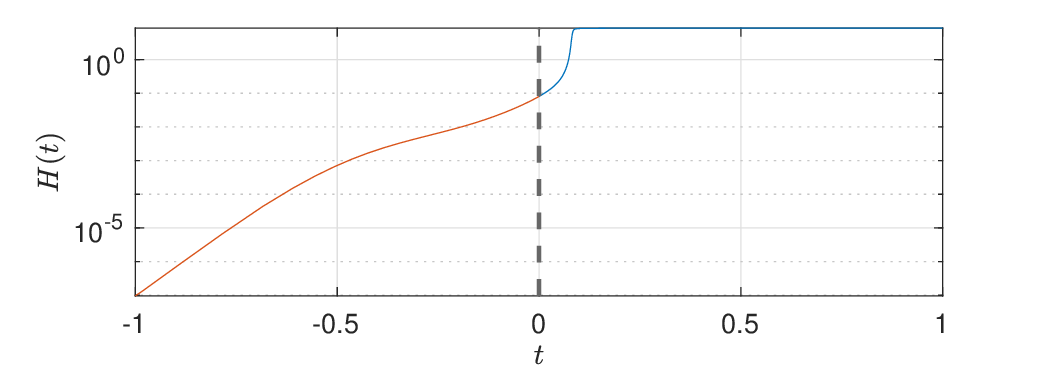} \\
 \caption{Profile of entropy evolution over positive and negative time.}
  \label{fig.100agents4}
\end{figure}

\begin{figure}[t]
   \center{  
 \includegraphics[width=9.5cm]{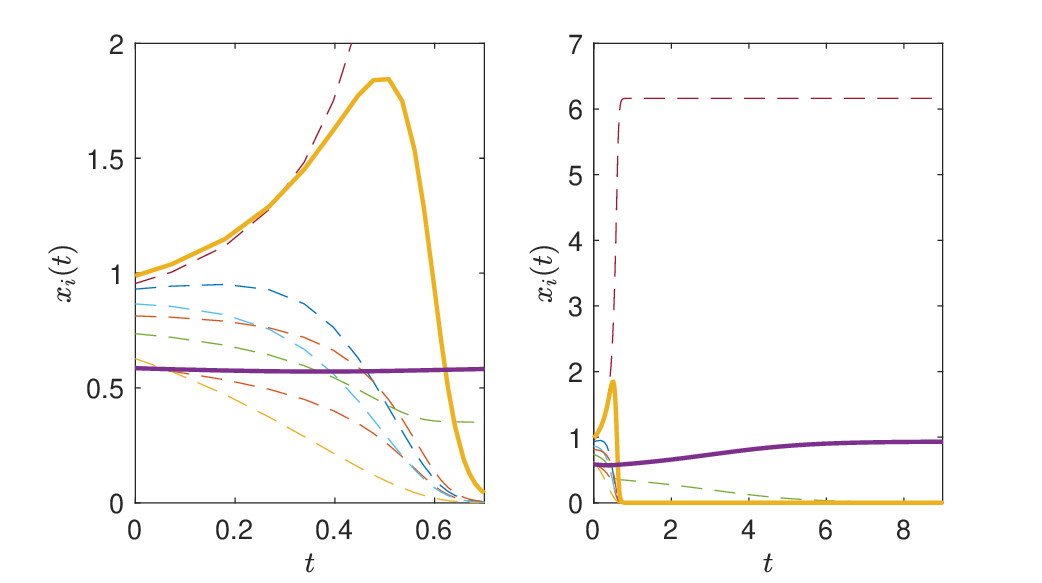} }
 \caption{The trajectories of nine agents over time. }
  \label{fig.9agents}
\end{figure}

\section{Simulation}
 \label{sec:simu}

We use numerical examples to illustrate the behavior of the model described by \eqref{model}. First, we consider a network of 100 agents with initial states randomly chosen in the range $[0,1]$ at $t=0$. The weighted adjacency matrix is set with $a_{ij} = 0$ or $a_{ij} = 1$. The model is simulated from $t=0$ to $t=1$, during which the winners-take-all phenomenon is observed, as shown in Fig.~\ref{fig.100agents1}. The top graph depicts the initial states of the agents, while the bottom graph shows the final states. At the final time, four agents emerge as winners, with the remaining agents' states having vanished, classifying them as losers.

To illustrate the reverse consensus model described by \eqref{model}, we run the simulation from $t=0$ to $t=-1$ in negative time. Fig.~\ref{fig.100agents3} demonstrates that consensus is achieved in negative time. The two graphs in Fig.~\ref{fig.100agents3} display the same trajectories of the 100 agents: the left graph provides a zoomed-in view for more detail, while the right graph shows the overall performance. To quantitatively describe the disorder among the agents, we define the entropy as
\begin{equation*}
H(t)= \frac{1}{n} \sum_{i\in\V} \left( x_i(t) -\frac{1}{n} \sum_{i\in \V}x_i(t)\right)^2.
\end{equation*}
This formula measures the dispersion of the agents' states from the average. Fig.~\ref{fig.100agents4} shows the profile of entropy evolution over positive and negative time. In positive time, the entropy increases to a constant, reflecting the resource seizing in competition that causes more disorder. In negative time, the entropy decreases to zero, indicating consensus. The plot is shown on a logarithmic scale to better highlight the trend as it approaches zero.

Interestingly, the initial values of the agents do not determine the winners; rather, the network topology plays a crucial role. To demonstrate this, we run a simulation with nine agents and plot the trajectories over time in Fig.~\ref{fig.9agents}. It is observed that the agent with the maximum initial value becomes a loser, while the agent with the minimum initial value emerges as a winner. Again, the two graphs in this figure display the same trajectories to better illustrate the details and overall performance.

We denote the agent with the minimum initial value as $\alpha$, whose initial and final values are $(0.59, 0.93)$. Next, we conduct a more extensive simulation by varying the weights $a_{\alpha j}$, where $j\in \V \setminus \{\alpha\}$, to either $0$ or $1$, and selecting the initial value $x_\alpha(0)$ from the range $[0,1.5]$. The settings for the other agents remain unchanged.

There are $2^8$ possible choices of neighbors (or called opponents) for agent $\alpha$, and the corresponding final values are plotted in Fig.~\ref{fig.optimal} for each initial value. The pair $(0.59, 0.93)$, corresponding to Fig.~\ref{fig.9agents}, is marked as $*$ in the figure. The bar in the figure specifies the range of initial values for the other eight agents. The two dotted lines indicate the sum of the agents' values and the relationship ``final value = initial value." It is observed that the initial value plays a significant role in the agent's success, while the choice of opponents is also critical. For instance, when the initial value is $x_\alpha(0) = 1.0$, the final value can result in two extremes: either the agent becomes a loser or emerges as the absolute winner, gaining the sum of all agents' values, i.e., the sole winner of the competition.

This simulation can be formulated as the following optimization problem:
\begin{align}
& \max_{a_{\alpha j} \in\{0,1\}, j\in \V \setminus \{\alpha\}}   x_{\alpha}(T) \nonumber \\
\text{subject to}\; & \text{Eq.~\eqref{model}}\; \text{with fixed} \nonumber \\
& x(0) \in \R_+^n,  \nonumber \\
& a_{i,j},\; i,j \in \V \setminus \{\alpha\} \label{opti}
\end{align}
for a fixed final time $T > 0$. The solution to this optimization problem shows how agent $\alpha$ can select its opponents to maximize its final value when entering a competitive network. The problem \eqref{opti} was solved in the above example using an exhaustive search, as shown in Fig.~\ref{fig.optimal}. However, the problem becomes more complex as the network size increases, additional constraints are introduced, and the problem is extended to multiple agents, such as $\alpha$, $\beta$, $\gamma$, etc.

In the case of optimization involving multiple agents, an algorithm to find a Nash equilibrium may be necessary. These challenges remain unresolved until we better understand the explicit mechanism by which the network topology, combined with the agents' initial values, determines the winners. This is a compelling direction for future research.
 
\begin{figure}[t]
   \center{  
 \includegraphics[width=7cm]{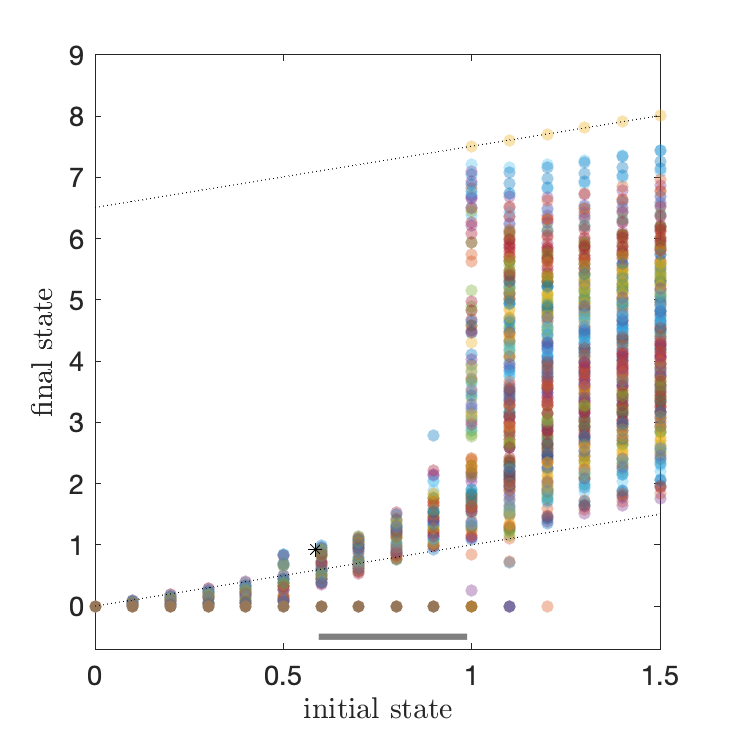} }
 \caption{Final state vs. initial state of agent $\alpha$ for various choices of opponents.}
  \label{fig.optimal}
\end{figure}
 
 \section{Conclusion}
 
 \label{sec:conclusion}

In this paper, we introduced a continuous-time multi-agent model with simple nonlinear dynamics that captures the winners-take-all behavior in a zero-sum game scenario. The model features equilibria where states of winners and losers are clearly defined. Specifically, loser states completely vanish while winner states persist. When winners do not directly compete, multiple winners can exist. These results are supported by rigorous theoretical analysis. Becoming a winner depends not only on the initial state but also on the network topology. The mechanisms through which network topology and initial agent values influence the determination of winners present intriguing avenues for future research. Additionally, investigating strategies for designing network topologies to achieve desired outcomes remains a challenging and important question for future studies.

\bibliographystyle{ieeetr}

\bibliography{Reference}

\end{document}